\documentclass[twoside,11pt]{article}
\usepackage{DGQ}
\usepackage{array}
\usepackage{mathpartir}
\usepackage{tikz-cd}
\usepackage{lmodern}
\usepackage[titletoc,title]{appendix}
\usepackage{stmaryrd}
\usepackage{tikz}

\begin{document}

\begin{titlepage}

\title{\textbf{Homotopy theory in a quasi-abelian category}}
\author{James Wallbridge\\
Kavli IPMU (WPI), UTIAS, University of Tokyo\\
5-1-5 Kashiwanoha, Kashiwa, Chiba 277-8583, Japan\\
and\\
Hitachi Central Research Laboratory\\
1-280 Higashi-Koigakubo, Kokubunji, Tokyo 185-8601, Japan\\
\\
james.wallbridge.vf@hitachi.com}
\date{\today}
\maketitle

\begin{center}
\textbf{Abstract}
\end{center}

We prove that the category of dg-modules and dg-algebras in a Grothendieck quasi-abelian category are endowed with a Quillen model structure.  This allows some flexibility in setting up a theory of derived algebraic geometry in the infinite dimensional setting.  For example, the category of complete bornological vector spaces, or equivalently, convenient vector spaces, is a Grothendieck quasi-abelian category.  Closely related is the Grothendieck quasi-abelian category of ind-Banach spaces whose associated model category is shown to be Quillen equivalent.  Applications include the Chevalley-Eilenberg resolution and the Koszul resolution of a commutative monoid object in a Grothendieck quasi-abelian category.  These can be used for the calculation of derived quotients by an infinite dimensional Lie algebra and derived intersections respectively.

\tableofcontents

\end{titlepage}

\section*{Introduction}

It has been appreciated for some time now that the correct setting in which to undertake the quantization of field theories is within the realm of homotopical algebra.  In particular, the construction of (on-shell) gauge invariant observables should be understood cohomologically.  Doing so clarifies the roles played by seemingly exotic structures in the theory such as ghost fields and antifields.  Working homotopically also enables one to properly understand the dependence on certain structures, such as the contractibility of the space of gauge fixing conditions, and prove universal properties.

In order to undertake this program rigorously, one should work in an appropriate $\infty$-category.  For full control over this $\infty$-category, it is advantageous that it arises as the localization of a closed model category.  In this case many constructions such as limits and colimits can be induced from the underlying model category and in many cases proofs simplify considerably.  Since collections of objects in field theories form infinite dimensional spaces, we wish to develop a homotopy theory of infinite dimensional spaces.  

The model category we are interested in for the study of field theories in this article is the category of chain complexes of complete bornological vector spaces, or equivalently, chain complexes of convenient vector spaces, with a model structure whose weak equivalences are refined quasi-isomorphisms.  Complete bornological and convenient vector spaces are two possible approaches to infinite dimensional spaces whose theory has now reached a level of maturity \cite{M2,KM}.  

There are other categories of vector spaces one may wish to consider.  They sit naturally in a diagram of adjunctions
\[
\begin{tikzcd}
\Ind(\tu{SNorm}_k)  \arrow[r,"{\colim}",yshift=0.4mm] \arrow[d,xshift=0.4mm] &\Born_k  \arrow[l,"\diss",shift left] \arrow[r,"\gamma",yshift=0.4mm]\arrow[d,xshift=0.4mm] &\tu{TVS}_k \arrow[l,"{\tu{vN}}",shift left] \arrow[r,"{t^\infty}",yshift=0.4mm]\arrow[d,xshift=0.4mm] & \BTVS_k \arrow[l,"i",shift left]\arrow[d,xshift=0.4mm]     \\
\Ind(\tu{Norm}_k)   \arrow[r,"{\colim}",yshift=0.4mm]\arrow[d,xshift=0.4mm]\arrow[u,shift left]  &\tu{SBorn}_k   \arrow[l,"\diss",shift left] \arrow[r,"\gamma",yshift=0.4mm] \arrow[d,xshift=0.4mm]\arrow[u,shift left] &\tu{STVS}_k     \arrow[l,"{\tu{vN}}",shift left] \arrow[r,"{s^\infty}",yshift=0.4mm]\arrow[d,xshift=0.4mm]\arrow[u,shift left]  & \tu{BSTVS}_k \arrow[l,"i",shift left]\arrow[d,xshift=0.4mm]\arrow[u,shift left] \\
\Ind(\Ban_k)   \arrow[r,"{\colim}",yshift=0.4mm]\arrow[u,shift left]  &\tu{CBorn}_k   \arrow[l,"\diss",shift left] \arrow[r,"\gamma",yshift=0.4mm]\arrow[u,shift left]  &\tu{CTVS}_k  \arrow[l,"{\tu{vN}}",shift left] \arrow[r,"{c^\infty}",yshift=0.4mm]\arrow[u,shift left]   & \Conv_k \arrow[l,"i",shift left]\arrow[u,shift left] 
\end{tikzcd}
\]
originating from the category $\Born_k$ of bornological vector spaces and the category $\tu{TVS}_k$ of locally convex topological vector spaces over the field $k$ of real or complex numbers.  

The vertical downward pointing arrows are separation and completion functors respectively and their right adjoints are inclusions.  The categories in the left hand column are ind-categories of semi-normed, normed and Banach spaces respectively and are related to the other categories through the dissection functor $\textit{diss}$.  The categories in the right hand column are the essential images of the left adjoint $\gamma$ to the functor $\tu{vN}$ which associates to a locally convex (resp. separated locally convex, complete locally convex) topological vector space its von Neumann bornology.  

One convenient arena in which to do homotopical algebra is Grothendieck abelian categories.  However, the examples we are interested in, in particular those in the diagram above, are not abelian.  Nevertheless, there does exist a weaker notion in which our examples do reside and whose theory we will exploit.  This is the theory of \textit{Grothendieck quasi-abelian} categories.   

One may wish to start with one of the categories in the third column to set up a theory of differential graded infinite dimensional vector spaces.  It turns out that these categories have a number of shortcomings which hinder their applicability.  In particular, they are not closed symmetric monoidal categories (with respect to the (separated, complete) projective tensor product).   

All of the categories in the diagram are pre-abelian.  The categories in the first two columns satisfy the conditions to be quasi-abelian \cite{Sc}.  Moreover, the quasi-abelian categories in the first two columns satisfy the key technical property of being locally presentable with their class of strict monomorphisms being closed under small filtered colimits.  We call this property \textit{Grothendieck} as it mimics the standard definition of a Grothendieck abelian category to the quasi-abelian setting.

The categories in the final column are not quasi-abelian.  The category $\tu{BTVS}_k$ of bornological topological vector spaces is only semi-abelian (in the sense of Palamodov \cite{Pa}) owing to a counterexample contained in \cite{BD}, and $\tu{BSTVS}_k$, the category of bornological separated topological vector spaces, is not even semi-abelian \cite{SW}.  Nevertheless, the final column have induced model structures on their categories of chain complexes arising from those in the second column.  In particular, there is a Quillen equivalence 
\[   \dg^{\CBorn_k}\rightleftarrows\dg^{\Conv_k}  \]
of chain complexes of complete bornological vector spaces and chain complexes of convenient vector spaces.  

Instead of proving the existence of a model category structure on the category of chain complexes in some of these quasi-abelian categories separately, we will prove a general existence theorem for the category of chain complexes in any Grothendieck quasi-abelian category.  All our examples of interest satisfy these conditions but this result may also be useful for other quasi-abelian categories not included in the diagram above.  Our model categories will then be shown to have the advantages of being closed symmetric monoidal.  

In analogy with the abelian case, we define both an injective and projective model structure.  In each of these cases the weak equivalences are refined quasi-isomorphisms.  These are morphisms which are isomorphisms on the abelian envelope of the internal cohomology.  Using this notion of weak equivalence we prove a Quillen equivalence
\[     \dg^{\Ind(\SNorm_k)}\rightleftarrows\dg^{\Born_k} \]
of model categories.  Here $\dg^C$ denotes the model category of chain complexes in a Grothendieck quasi-abelian category $C$.  This result holds for the subcategories in the other horizontal lines in the first two columns of the above diagram. 

Under certain conditions, the category of commutative differential graded algebras in a Grothendieck quasi-abelian category can be endowed with a model structure.  This structure is needed to define (affine) derived stacks in the infinite dimensional setting which is useful for understanding the kinds of moduli problems arising in field theory.  We prove a related result for dg-Lie algebras in the quasi-abelian setting which is used to show that the cofibrant replacement of a commutative monoid object in a Grothendieck quasi-abelian category can be given by a Chevalley-Eilenberg resolution.  We prove that another cofibrant replacement is given by a Koszul resolution in a different model category.  As above, there is a chain of Quillen equivalences
\[     \cdga^{\Ind(\SNorm_k)}\rightleftarrows\cdga^{\Born_k}  \]
of model categories where $\cdga^C$ denotes the model category of commutative differential graded algebras in a Grothendieck quasi-abelian category $C$.

\section*{Relation to other work }
An injective model structure on the category of (bounded) chain complexes in a quasi-abelian category was already constructed in \cite{Bu}.  We will apply stronger conditions on a quasi-abelian category in order to produce a \textit{combinatorial} model structure making it suitable to study module and algebra objects efficiently.  

Recent work of \cite{Ke} constructs model category structures on certain exact categories which coincide with our own  (although the proofs are different).  However, our presentation differs from \textit{loc.cit.} in that we emphasize the specific properties arising when the category in question is Grothendieck.  All the examples we have in mind satisfy this stricter assumption so it makes sense to exploit the general theory, and the simplifications it affords, in this context.

\section*{Notation}

In this paper we will often be confronted with set theoretic issues.  In order not to burdon the notation, we will take the pragmatic approach of fixing here a Grothendieck universe $\bb{U}$ and calling elements therein \textit{$\bb{U}$-small}.  We then fix universes $\bb{V}$ and $\bb{W}$ such that $\bb{U}\in\bb{V}\in\bb{W}$ and refer to elements in $\bb{V}$ as \textit{$\bb{V}$-small} or \textit{large} and those in $\bb{W}$ as \textit{very large}.  We will leave it to the reader to supplement the terms small limits, small colimits and locally presentable category to $\bb{U}$-small limits, $\bb{U}$-small colimits and $\bb{U}$-locally presentable category and likewise for $\bb{V}$ and $\bb{W}$.  We also assume the Vopenka principle which ensures that a full reflective subcategory of a locally presentable category is itself locally presentable.

\section*{Acknowledgements}

Special thanks is due to Jack Kelly and Federico Bambozzi for feedback on a preliminary version of this text.  I would also like to thank Bertrand To\"en for comments.  This work was partially supported by the World Premier International Research Center Initiative (WPI), MEXT, Japan.

\section{Grothendieck quasi-abelian categories}\label{qac}

Let $k$ be the field of real or complex numbers.  Recall that a convex bornological vector space is a $\bb{U}$-small vector space $V$ endowed with a bornology $\sB_V$ (which we will often omit from the notation) whose elements are called bounded subsets \cite{H-N}.  A convex bornological vector space is said to be separated if all bounded disks are norming and complete if each bounded subset is contained in a complete bounded disk.  

We will denote the large category of convex bornological vector spaces and bounded linear maps by $\Born_k$ and the full subcategories of separated and complete objects by $\tu{SBorn}_k$ and $\tu{CBorn}_k$ respectively.  The latter is equivalent to the large category $\tu{CTVS}_k$ of complete locally convex topological vector spaces over $k$ with bounded (as opposed to continuous) linear maps.  All bornological vector spaces will be henceforth convex so objects in $\Born_k$ will be simply called \textit{bornological vector spaces}.  For any object $V$ in $\Born_k$, the canonical morphism
\[      \colim_{B\in\sD_V} V_B\ra V   \]
is an isomorphism where $V_B$ is the linear hull of $B$ endowed with the gauge semi-norm and $\sD_V$ is the collection of disks in $V$.  

Let $(V,\sB_V)$ and $(W,\sB_W)$ be bornological vector spaces.  A subset $S$ of the vector space $\Hom(V,W)$ is called equibounded if $\{f(x)|f\in S,x\in B\}\in\sB_W$ for all $B\in\sB_V$.  The bornology of equibounded maps turns the category of bornological (resp. separated bornological, complete bornological) vector spaces into a cartesian closed category.  We denote by $\uHom(V,W)$ the internal hom object between $V$ and $W$.

Let $\tu{TVS}_k$ denote the large category of locally convex $\bb{U}$-small topological vector spaces, hereafter called \textit{topological vector spaces}, over $k$ and continuous linear maps.  Let
\[   \tu{\tu{vN}}:\tu{TVS}_k\ra\Born_k   \]
denote the functor which associates to $M$ its von Neumann bornology where a subset is bounded if and only if it is absorbed by every neighborhood of the origin in $M$.  The bornological vector space $\tu{vN}(M)$ is often called the \textit{bornologification} of $M$.  There exists a fully faithful left adjoint which associates to $V$ a topological vector space $\gamma(V)$ with basis those subsets of the neighborhood of the origin that absorb bounded sets.  This is the finest topology whose von Neumann bornology coincides with the original one, or equivalently, the unit map of the adjunction is bounded.  

Therefore, we define a \textit{bornological topological vector space} to be a topological vector space $M$ such that $\gamma\circ \tu{vN}(M)\simeq M$ is an isomorphism.  Denote the essential image of $\gamma$ by $\tu{BTVS}_k$.  The essential image of $\gamma$ on the subcategory $\tu{SBorn}_k$ of separated bornological vector spaces will be denoted $\tu{BSTVS}_k$ and on the subcategory $\tu{CBorn}_k$ by $\Conv_k$.  The objects in $\BSTVS_k$ will be called \textit{bornological separated topological vector spaces} and objects in $\Conv_k$, \textit{convenient vector spaces}.  See \cite{KM} for the theory of convenient vector spaces.

Recall that there exists a \textit{dissection functor}
\[     \tu{diss}:\Born_k\ra\Ind(\tu{SNorm}_k)  \]
which sends a bornological vector space to an inductive system of $\bb{U}$-small semi-normed spaces as follows.  Let $V$ be a bornological vector space and $(\sD_V,\leq)$ the directed set of bounded disks $B$ in $V$ partially ordered by absorption.  Then $\tu{diss}(V):=(V_B)_{B\in\sD_V}$ is an inductive system of semi-normed spaces.  The dissection functor is fully faithful.  

There exists a left adjoint to the dissection functor which sends an inductive system $(V_B)_{B\in\sD(V)}$ to $\colim_{B\in\sD_V}V_B$.  The dissection functor on the subcategory of separated objects defines a full embedding into the category $\Ind(\tu{Norm}_k)$ of $\bb{V}$-small ind-objects in the large category $\tu{Norm}_k$ of $\bb{U}$-small normed spaces and bounded linear maps.  Moreover, on the subcategory of complete objects, it defines a full embedding into the category $\Ind(\Ban_k)$ of ind-objects in the large category $\Ban_k$ of Banach spaces when we replace the directed set $(\sD_V,\leq)$ by the directed set of complete bounded disks in $V$.  The essential image of these dissection functors consists of inductive systems $(V_i,\alpha_i)$ such that each $\alpha_i$ is a monomorphism.

All of the examples above share the following underlying structure~:

\begin{dfn}
A \textit{pre-abelian} category is an additive category with kernels and cokernels.  
\end{dfn}

It follows that a pre-abelian category admits finite limits and colimits.  In any category with kernels and cokernels, a morphism is said to be \textit{strict} if its coimage is isomorphic to its image, ie. a morphism $f:x\ra y$ is strict if the induced morphism
\[  \tu{Coker}(\ker(f))=:\tu{Coim}~f\xra{\tilde{f}}\tu{Im}~f:=\tu{Ker}(\coker(f))    \]
is an isomorphism.  In any pre-abelian category, pullbacks preserve strict monomorphisms (kernels) and pushouts preserve strict epimorphisms (cokernels).  Allowing strict epimorphisms (resp. strict monomorphisms) to be stable under pullbacks (resp. pushouts) leads to the following definition \cite{Sc}.

\begin{dfn}
A pre-abelian category is said to be \textit{quasi-abelian} if its class of strict monomorphisms is stable under pushouts along arbitrary morphisms and its class of strict epimorphisms is stable under pullback along arbitrary morphisms.  
\end{dfn}

Therefore, in a quasi-abelian category, all strict monomorphisms and all strict epimorphisms are stable under pullbacks and pushouts.  Moreover, every morphism in a quasi-abelian category has a canonical decomposition into a strict epimorphism (resp. epimorphism) followed by a monomorphism (resp. strict monomorphism).     

\begin{ex}
The quasi-abelian category $\Born_k$ of bornological vector spaces is quasi-abelian \cite{PS}.  In this category, strict monomorphisms correspond to bornological isomorphisms onto their image with the subspace bornology.  On the other hand, strict epimorphisms are maps such that any bounded subset in the codomain is the image of a bounded subset.  
\end{ex}

A quasi-abelian category is said to be \textit{semi-abelian} if for every morphism $f$, the induced morphism $\tilde{f}$ is both an epimorphism and a monomorphism.  Every quasi-abelian category is semi-abelian \cite{Sc}.  However, not every strict monomorphism or strict epimorphism is necessarily stable under pushouts or pullbacks respectively.

\begin{ex}
The pre-abelian category $\BTVS_k$ of bornological topological vector spaces is semi-abelian \cite{Si} but not quasi-abelian.  The authors in \cite{BD} show that strict epimorphisms are not stable under pullback.
\end{ex}

In order to prove that there exists a model structure on the category of chain complexes in a quasi-abelian category we will need to introduce some further assumptions.  An abelian category is said to be Grothendieck if it is locally presentable and its collection of monomorphisms is stable under filtered colimits.  The analogue in our case is closure under filtered colimits of strict monomorphisms.  Therefore we make the following definition.

\begin{dfn}\label{groth}
Let $C$ be a quasi-abelian category.  Then $C$ is said to be \textit{Grothendieck} if it is locally presentable and its collection of strict monomorphisms is closed under small filtered colimits.
\end{dfn}

\begin{lem}\label{modisgqa}
Let $C$ be a (symmetric) monoidal Grothendieck quasi-abelian category and $R$ a (commutative) monoid object in $C$.  Then the category $\Mod_R(C)$ of $R$-modules in $C$ is a Grothendieck quasi-abelian category.
\end{lem}

\begin{proof}
For any locally presentable category $C$, the category of modules over any commutative monoid object in $C$ is locally presentable.  The result now follows from the fact that the forgetful functor from $\Mod(R)$ to $C$ preserves all limits and colimits.  
\end{proof}

An important class of examples comes from ind-objects in quasi-abelian categories and subcategories thereof.  Recall that a subcategory is said to be \textit{reflective} if the natural inclusion admits a left adjoint.

\begin{prop}\label{ind}
Let $C$ be a finitely cocomplete quasi-abelian category with exact filtered colimits.  Then the following hold~:
\begin{enumerate}
\item The category $\Ind(C)$ of ind-objects in $C$ is a Grothendieck quasi-abelian category.
\item Any full reflective subcategory of $\Ind(C)$ is a Grothendieck quasi-abelian category.
\end{enumerate}
\end{prop}

\begin{proof}
Let $C$ be a category satisfying the assumptions of the proposition.  Then the category $\Ind(C)$ is quasi-abelian and cocomplete.  Any cocomplete category of the form $\Ind(C)$ is locally presentable by definition.  Closure under small filtered colimits of strict monomorphisms follows from the assumption on $C$.  Indeed, if filtered colimits are exact in $C$ they are exact in $\Ind(C)$.  Let $\Ind(C)^{\Delta^1}$ denote the category of morphisms in $\Ind(C)$ and consider the colimit functor 
\[  \colim_I:(\Ind(C)^{\Delta^1})^I\ra\Ind(C)^{\Delta^1}  \] 
for a filtered category $I$.  Since filtered colimits are exact in $\Ind(C)^{\Delta^1}$ we have a diagram
\[
\begin{tikzcd}
\tu{Coim}(\colim_IF) \arrow[r]\arrow[d]  &\tu{Im}(\colim_I F)\arrow[d]\\
\colim_I\tu{Coim}(F)  \arrow[r]  &\colim_I\tu{Im}(F)
\end{tikzcd}
\]
of objects in $\Ind(C)$ for $F$ in $(\Ind(C)^{\Delta^1})^I$ where the vertical arrows are equivalences.  Since the bottom arrow is an equivalence by assumption, the top horizontal arrow is an equivalence and thus the collection of strict monomorphisms in $\Ind(C)$ is closed under filtered colimits as required.

By Proposition~1.39 of \cite{AR}, any full reflective subcategory of a locally presentable category which is closed under filtered colimits is itself locally presentable.  Since left adjoints preserve colimits and $\Ind(C)$ is Grothendieck quasi-abelian, any full reflective subcategory is likewise Grothendieck quasi-abelian.
\end{proof}

All of the ind-categories of interest to us satisfy the conditions of Definition~\ref{groth}.  Recall that $\SNorm_k$ is the category whose objects are semi-normed spaces $(V,\rho_V)$ and a morphism between $(V,\rho_V)$ and $(W,\rho_W)$ is a morphism $f:V\ra W$ of vector spaces such that $|\rho_W\circ f|\leq c\rho_V$ for some $c>0$.  Equivalently, a morphism is a continuous morphism of locally convex topological vector spaces for the canonical topology induced by the semi-norm.  We will emit the semi-norm from the notation and refer simply to $V$.

\begin{prop}\label{snormgqa}
The categories $\Ind(\SNorm_k)$, $\Ind(\Norm_k)$ and $\Ind(\Ban_k)$ are Grothendieck quasi-abelian.
\end{prop}

\begin{proof}
The category $\SNorm_k$ is quasi-abelian by Proposition~3.2.4 of \cite{Sc}.  We now verify the conditions of Proposition~\ref{ind} and show that filtered colimits are exact in $\SNorm_k$.  We will show that the functor $\colim_I:(\SNorm_k)^I\ra\SNorm_k$ preserves monomorphisms for a filtered category $I$ and leave the remaining steps to the reader.  Let $\alpha:F\ra G$ be a monomorphism in $(\SNorm_k)^I$ and $v\in\colim_IF$ such that $\colim_I(\alpha)(v)=0$ in $\colim_IG$.  We need to show that $v=0$.  We know that $v$ is the image of some $v_i\in F(i)$ and therefore $\alpha_i(v_i)\in G(i)$ vanishes in $\colim_IG$.  Therefore there exists a map $u:i\ra j$ such that $G(u)(\alpha_i(v_i))=\alpha_{j}(F(u)(v_i))=0$ in $G(j)$.  Since $\alpha_j$ is a monomorphism, then $F(u)(v_i)=0$ in $F(j)$ and $v=0$ in $\colim_IF$ as required.  Since the category $\SNorm_k$ is cocomplete, the category $\Ind(\SNorm_k)$ is Grothendieck quasi-abelian by Proposition~\ref{ind}.

The categories $\Norm_k$ and $\Ban_k$ are quasi-abelian since they are, by definition, full additive subcategories of $\SNorm_k$ and by Proposition~3.2.16 of \cite{Sc} admit kernels and cokernels.  Since $\Norm_k$ and $\Ban_k$ have finite colimits, it follows from Proposition~\ref{ind} that $\Ind(\Norm_k)$ and $\Ind(\Ban_k)$ are locally presentable.  They are moreover Grothendieck since the separation and completion functors are left adjoints and thus preserve colimits.  
\end{proof}

\begin{prop}\label{bornisgqa}
The categories $\Born_k$, $\SBorn_k$ and $\CBorn_k$ are Grothendieck quasi-abelian categories.
\end{prop}

\begin{proof}
The category of bornological (separated bornological, complete bornological) vector spaces over $k$ is quasi-abelian by Proposition~1.8 (resp. Proposition~4.10, Proposition~5.6) of \cite{PS}.  From Proposition~\ref{snormgqa}, the categories $\Ind(\SNorm)_k$, $\Ind(\Norm_k)$ and $\Ind(\Ban_k)$ are locally presentable and the dissection functor is fully faithful.  By Proposition~1.9 (resp. Proposition~4.12, Proposition~5.6) of \cite{PS}, the category $\Born_k$ (resp. $\SBorn_k$, $\CBorn_k$) is cocomplete.  Every full reflective subcategory of a locally presentable category which is closed under colimits is locally presentable. 
Therefore $\tu{Born}_k$, $\SBorn_k$ and $\CBorn_k$ are locally presentable quasi-abelian.  These categories are moreover Grothendieck since the dissection functor is right adjoint and thus its left adjoint preserves colimits.
\end{proof}

A final example, important for applications, is the category of Fr\'echet spaces.  Let $\tu{Fr\'e}_k$ denote the full additive subcategory of $\CTVS_k$ spanned by Fr\'echet spaces.

\begin{prop}
The category $\tu{Fr\'e}_k$ is Grothendieck quasi-abelian.
\end{prop}

\begin{proof}
The category $\tu{Fr\'e}_k$ of Fr\'echet spaces is additive by definition.  For any morphism $f:V\ra W$ in $\tu{Fr\'e}$, the kernel of $f$ is the subspace $f^{-1}(0)$ of $V$ endowed with the induced topology and the cokernel of $f$ is the quotient space $W/\overline{f(V)}$ endowed with the quotient topology.  Therefore, the category $\tu{Fr\'e}_k$ is pre-abelian.  It follows from Proposition~4.4.5 of \cite{P2} that it is quasi-abelian.  The bornologification functor $\tu{vN}:\CTVS_k\ra\CBorn_k$ is fully faithful on the subcategory of Fr\'echet spaces.  Therefore the composition of functors $\tu{diss}\circ\tu{vN}$ exhibits $\tu{Fr\'e}_k$ as a full reflective subcategory of the locally presentable category $\Ind(\Ban_k)$ and thus by Proposition~\ref{ind}, $\tu{Fr\'e}_k$ is locally presentable.
\end{proof}

\section{Differential graded modules}\label{dgm}

In this section we will set up the homotopy theory of infinite dimensional differential graded vector spaces.  We will use the category of complete bornological vector spaces as an example for illustration.  We first define a general notion of quasi-isomorphism for a quasi-abelian category.

Let $(M,d)$ be a cochain complex of complete bornological vector spaces.  Since the category of complete bornological vector spaces is quasi-abelian, it makes sense to consider the quotient vector space $\tu{Ker}(d^{n+1})/\tu{Im}(d^{n})$ with the quotient bornology.  It is then tempting to define a cohomology functor sending $M$ to this quotient.  However, this quotient bornology is not necessarily complete.  What does make sense is the following.  Let $V$ be a complete bornological vector space and $W$ a subspace of $V$.  Then $W$ is said to be \textit{closed} if limits of sequences in $W$ which converge in $V$ belong to $W$.  The \textit{closure} $\overline{W}$ of $W$ is the intersection of all the closed subspaces of $V$ containing $W$.  The kernel $\tu{Ker}(d^n)$ is complete.  Therefore it follows from \cite{H-N} that the quotient vector space $\tu{Ker}(d^{n+1})/\overline{\tu{Im}}(d^{n})$ with the quotient bornology is complete.  Although this latter notion of cohomology is natural, it only applies then the image of $d_n$ has closed range.  The remedy, in general, is to use the abelian envelope of a quasi-abelian category.

Let $C$ be a quasi-abelian category.  Then there exists a pair $(\cA(C),a)$ consisting of an abelian category $\cA(C)$ together with a full embedding 
\[  a:C\ra\cA(C)  \] 
called the \textit{abelianization functor} satisfying the following universal property~: for any abelian category $D$, the functor $a$ induces an equivalence
\[    \uHom^{lex}(\cA(C),D)\ra\uHom^{lex}(C,D)   \]
of categories where the superscript \textit{lex} refers to left exact functors.  By the enriched Yoneda lemma, an explicit construction of $\cA(C)$ is given by the abelian category $\cA(C)=\uHom^{lex}(C,\Ab)$ of left exact functors where $\Ab$ is the category of abelian groups.  The category $\cA(C)$ is called the \textit{abelian envelope} of $C$.  

Let $C$ be a quasi-abelian category and $\dg^C$ the category of cochain complexes in $C$.  We define a cohomology functor
\[     H^n:\dg^C\ra C   \]
by sending $M$ to $\coker(\tu{Im}~d^n\ra\tu{Ker}~d^{n+1})$ in $C$.  Using the abelian envelope construction, we have a second cohomology functor
\[    \cH^n:\dg^C\ra\cA(C)   \]
sending $M$ to $H^n(a_*(M))$ in the abelian category $\cA(C)$ where $a_*:\dg^C\ra\dg^{\cA(C)}$ is the induced functor.  We call this functor the \textit{refined cohomology}.

\begin{dfn}\label{qi}
Let $C$ be a quasi-abelian category.  Then a morphism in $\dg^C$ is said to be a \textit{refined quasi-isomorphism} if it is an isomorphism on refined cohomology.
\end{dfn}

Note that when $M$ is a strict complex, ie. all the maps in the complex are strict, then there exists an isomorphism
\[   \cH^n(M)\simeq a(H^n(M))   \]
in $\cA(C)$ between the refined cohomology and the abelianization of the cohomology.  

A morphism of cochain complexes in an abelian category is an isomorphism on cohomology if and only if its mapping cone is acyclic.  An analogous definition also applies to quasi-abelian categories.   Let $C$ be a quasi-abelian category.  A null sequence 
\[    V\xra{f} W\xra{g} X  \]
in $C$ is said to be \textit{strictly exact} if $f$ is strict and the canonical morphism 
\[  \tu{Im}(f)\ra\tu{Ker}(g)  \] 
is an isomorphism.  It is said to be \textit{strictly coexact} if $g$ is strict and the canonical morphism 
\[ \tu{Coker}(f)\ra\tu{Coim}(g)  \] 
is an isomorphism.  The sequence is said to be \textit{acyclic} if it is both strictly exact and strictly coexact.  A cochain complex in $C$ is said to be \textit{acyclic} if it is acyclic in each degree.  It follows from Proposition~2.39 of \cite{Ke} that a morphism in $\dg^C$ is a quasi-isomorphism if and only if its mapping cone is acyclic.

The following is a weaker version of the analogous statement for cochain complexes in a Grothendieck abelian category (see for example Proposition~3.13 of \cite{Be}).  The following notation will be used throughout this article.

\begin{notn}\label{notation}
Let $C$ be a quasi-abelian category.  For any morphism $u:V\ra W$ in $C$, let 
\[     C[u,n]:=\cdots\ra 0\ra V\xra{u} W\ra 0\ra\cdots  \]
be the complex in $\dg^C$ with $V$ in degree $n$ and $W$ in degree $(n+1)$ and zero otherwise. 
\end{notn}

\begin{prop}\label{injective}
Let $C$ be a Grothendieck quasi-abelian category.  There exists a left proper combinatorial model structure on the category $\dg^C$ of cochain complexes in $C$ with the following classes of morphisms~:
\begin{itemize}
\item[$(\sC)$] The cofibrations are the degreewise strict monomorphisms.
\item[$(\sW)$] The weak equivalences are the refined quasi-isomorphisms.
\item[$(\sF)$] The fibrations are the those maps with the right lifting property with respect to trivial cofibrations.
\end{itemize}
\end{prop}

\begin{proof}
We first need to construct a small set of generating cofibrations $\sC_0$ such that each cofibration belongs to the weakly saturated class generated by $\sC_0$.   Since $C$ is locally presentable we can define an object $V:=\bigoplus V_i$ of $C$ given by the coproduct of the objects $V_i$ which generate $C$ under small colimits.  For every strict monomorphism $u:W\ra V$, we define $\sC_0$ to be the collection of all strict monomorphisms $C[u,n]\ra C[\id_V,n]$ for all $n\in\bb{Z}$.  Arguing as in Proposition~1.3.5.3 of \cite{L2}, using the the fact that strict monomorphisms are stable under filtered colimits in view of $C$ being Grothendieck, we find that every cofibration belongs to the smallest weakly saturated class of morphisms containing $\sC_0$.  Conversely, it is clear that $\sC$ contains $\sC_0$ and is weakly saturated (in particular, strict monomorphisms are closed under the formation of pushouts by definition). 

We will now check the conditions of Proposition~A.2.6.13 of \cite{L1}.  The refined cohomology functor 
\[    \cH:\dg^C\ra\prod_{n\in\bb{Z}}\cA(C)  \]
commutes with filtered colimits since $a_*$ and $H:\dg^{\cA(C)}\ra\prod_{n\in\bb{Z}}\cA(C)$ commute with filtered colimits.  The class of isomorphisms in $\dg^{\cA(C)}$ is perfect since $\dg^{\cA(C)}$ is a locally presentable category, therefore it follows from Corollary~A.2.6.12 of \cite{L1} that the weak equivalences in $\dg^C$ are a perfect class. 

Now consider the pushout diagram
\[
\begin{tikzcd}
M   \arrow[r,"g"]\arrow[d,"f"] & N'  \arrow[d,"f'"]\\
N   \arrow[r,"g'"] &P
\end{tikzcd}
\]
in $\dg^C$ where $f$ is a cofibration and $g$ is a weak equivalence.  We need to show that $f':N\ra P$ is a weak equivalence.  Since $C$ is quasi-abelian, it follows that $g'$ is a  degreewise strict monomorphism.  It then follows from Proposition~2.12 of \cite{B1} that this diagram is part of a commutative diagram
\[
\begin{tikzcd}
M  \arrow[r,"g"]  \arrow[d,"f"]   & N'  \arrow[r,"h"] \arrow[d,"f'"]   & Q\arrow[d,"id"] \\
N    \arrow[r,"g'"]  & P  \arrow[r,"h'"] &Q
\end{tikzcd}
\]
where the horizontal rows are short exact with $h$ and $h'$ are strict epimorphisms.  The refined cohomology functor preserves short exact sequences and therefore, for each $n$, we have a diagram
\[
\begin{tikzcd}
0 \arrow[r]   & \cH^n(M)  \arrow[r,"g"]\arrow[d,"\theta_1"]  & \cH^n(N')  \arrow[r,"h"]\arrow[d,"\theta_2"] & \cH^n(Q) \arrow[r]\arrow[d,"\theta_3"] & 0 \\
0  \arrow[r]  &\cH^n(N)    \arrow[r,"{g'}"]       &\cH^n(P)  \arrow[r,"{h'}"]  &\cH^n(Q) \arrow[r] & 0
\end{tikzcd}
\]
where $\theta_1$ and $\theta_3$ are isomorphisms.  It follows from the short five lemma that $\theta_2$ is an isomorphism.

Finally, we need to show that if $f:M\ra N$ is a morphism in $\dg^C$ with the right lifting property with respect to every morphism in $\sC$, then $f\in\sW$.  We show that for any $n\in\bb{Z}$, the map $\alpha:\cH^n(M)\ra\cH^n(N)$ is an epimorphism (proving that $\alpha$ is a monomorphism can be done in an analogous way to that in Proposition~1.3.5.3 of \cite{L2}).  Let $K_M:=\ker(d^{n+1}_M)$, $K_N:=\ker(d_N^{n+1})$ and $u:0\ra K_N$ be the zero map.  Therefore $0\ra C[u,n]$ is a cofibration and the strict monomorphism $C[u,n]\ra N^{n+1}$ lifts to a map $K_N\ra K_M$.  Thus $\alpha$ is an epimorphism on refined cohomology.   
\end{proof}

We will call the model structure of Proposition~\ref{injective} the \textit{injective model structure}.  It is clear that every object in the injective model structure is cofibrant.  To characterize the fibrant objects, we recall some definitions.  

A functor between quasi-abelian categories is said to be \textit{exact} if it preserves short exact sequences.  An object $I$ in a quasi-abelian category $C$ is said to be \textit{injective} if the functor 
\[     \Hom_C(-,I):C^\circ\ra\Ab   \]
is exact.  The object $I$ is injective if and only if for every strict monomorphism $f:V\ra W$ in $C$, the induced map
\[   \Hom_C(W,I)\ra\Hom_C(V,I)  \]
is a surjection.  The category $C$ is then said to have \textit{enough injectives} if for every object $V$ in $C$, there exists a strict monomorphism $V\ra I$ with $I$ injective.

\begin{lem}
Let $C$ be a Grothendieck quasi-abelian category.  If $M$ is fibrant in the injective model structure on $\dg^C$, then each $M^n$ is an injective object in $C$.
\end{lem}

\begin{proof}
Let $u:V\ra W$ be a strict monomorphism in $C$ and consider the induced trivial cofibration $f:C[\id_V,n]\ra C[\id_W,n]$ in $\dg^C$ with respect to the injective model structure.  Now assume that $M$ is fibrant.  Then by definition, it has the right lifting property with respect to $f$ and thus $V\ra M^n$ factors through the strict monomorphism $u$.  It follows that $M^n$ is an injective object of $C$.  
\end{proof}

\begin{lem}
Let $C$ be a Grothendieck quasi-abelian category.  Then the category $C$ has enough injectives.  
\end{lem}

\begin{proof}
Let $u:V\ra 0$ be a morphism in $C$ and factor the map $C[u,0]\ra *$ by a trivial cofibration $C[u,0]\ra Q$ followed by a fibration in the injective model structure on $\dg^C$.  Therefore $Q$ is a fibrant object of $\dg^C$ and thus, by part one, $Q^0$ is an injective object of $C$.  It follows that for every object $V$ of $C$, there exists a strict monomorphism $V\ra Q^0$ for some injective object $Q^0$ and hence $C$ has enough injectives.
\end{proof}

We now give some examples of dg-modules in a quasi-abelian category which are simple extensions of familiar examples.

\begin{ex}\label{fine}
Let $\dg^{\Vect_k}$ denote the category of dg-vector spaces over $k$.  Then there exists a fully faithful functor
\[    \tu{Fine}_*:\dg^{\Vect_k}\ra\dg^{\CBorn_k}  \]
sending a dg-vector space $M$ to the dg-vector space $M$ endowed with the \textit{fine bornology}.  A subset $N$ of $M$ is bounded in $\tu{Fine}_*(M)$ if and only if there exists a finite dimensional subspace $(M_N)^n\subseteq M^n$ for each $n\in\bb{Z}$ such that $N^n$ is a bounded subset of $(M_N)^n\simeq k^n$.  It is the finest possible bornology on $M$.  The functor $\tu{Fine}$ is left adjoint to the forgetful functor and the adjunction is a Quillen adjunction of model categories where $\dg^{\Vect_k}$ is endowed with its injective model structure.
\end{ex}

\begin{ex}\label{borcpct}
Let $\dg^{\CTVS_k}$ denote the category of chain complexes of complete topological vector spaces over $k$.  Let 
\[  \tu{vN}_*:\dg^{\CTVS_k}\ra\dg^{\CBorn_k}  \]
denote the functor which sends a complete topological dg-vector space $M$ to its \textit{bornologification}, ie. $\tu{vN}_*(M)$ is endowed with the \textit{von Neumann bornology} in which for all $n\in\bb{Z}$, a subset $N^n\subset M^n$ is bounded if and only if it is absorbed by every neighbourhood of the origin in $M^n$.  This functor admits a fully faithful left adjoint $\gamma$ which associates to a bornological dg-vector space $V$, the complete topological dg-vector space $\gamma(V)$ such that for each $n\in\bb{Z}$, a basis of $\gamma(V)^n$ consists of those subsets of the neighborhood of the origin that absorb bounded sets.  This is the finest locally convex topology whose von Neumann bornology coincides with the original one.  The restriction of the bornologification functor to Fr\'echet dg-vector spaces 
\[  \tu{vN}_*:\dg^{\tu{Fre}_k}\ra\dg^{\CBorn_k}  \]
is fully faithful.  
\end{ex}

\begin{ex}
Another interesting bornology one can attach to an object in $\dg^{\CTVS_k}$ is the \textit{precompact bornology}.  Let
\[   \tu{pC}_*:\dg^{\CTVS_k}\ra\dg^{\CBorn_k}    \]
denote the functor which sends a complete topological dg-vector space $M$ to the dg-vector space $M$ endowed with the bornology where subset $N$ of $M$ is bounded if and only if for each $n\in\bb{Z}$, the closure of $N^n$ is compact.  Again, the restriction functor to Fr\'echet dg-vector spaces is fully faithful.
\end{ex}

One can construct cartesian closed categories of cochain complexes of topological vector spaces by exploiting their underlying bornological properties.  Likewise, although Proposition~\ref{injective} does not define an injective model structure on categories of topological vector spaces directly, we can transfer the injective model structure on categories of bornological vector spaces to reflective subcategories of topological vector spaces via the completion functors.

\begin{prop}
There exists a Quillen equivalence 
\[     \dg^{\Born_k}\ra\dg^{\BTVS_k}  \]
of model categories.  This Quillen equivalence holds between separated bornological vector spaces and bornological separated topological vectors spaces and between complete bornological vector spaces and convenient vector spaces.
\end{prop}

\begin{proof}
From the notation in the introduction, recall that there exists an adjunction $t^\infty\circ\gamma\dashv\tu{vN}\circ i$ inducing an equivalence of categories where $t^\infty$ is the completion functor.  This induces an equivalence of categories 
\[  t_*^\infty\circ\gamma_*:\dg^{\Born_k}\ra\dg^{\BTVS_k}  \]
of cochain complexes where $\gamma_*$ and $t_*^\infty$ are the induced dg-functors.  The category $\dg^{\Born_k}$ is endowed with an injective model structure by Proposition~\ref{bornisgqa} and Proposition~\ref{injective}.  

We now endow $\dg^{\BTVS_k}$ with the transfer model structure.  A functor inducing an equivalence of categories induces a Quillen equivalence via the transfer model structure.  The same argument holds for the other two cases.
\end{proof}

Let $C$ be a (symmetric) monoidal quasi-abelian category.  We let 
\[  \dg_R^C:=\tu{Ch}(\Mod_R(C))  \]
denote the category of chain complexes in the category of modules over a monoid object $R$ in $C$.  We call $\dg_R^C$ the category of \textit{dg-modules} over $R$.

There exists a \textit{bornological tensor product} satisfying the following universal property~:  there exists a bornological isomorphism
\[    \uHom(V\otimes W,X)\xras\uHom(V\times W,X)   \]
for $V,F,X\in\Born_k$.  The \textit{complete bornological tensor product} is the completion of the bornological tensor product.  The category $\CBorn_k$ is a symmetric monoidal category for the complete bornological tensor product.  A commutative monoid object in $\CBorn_k$ will be called a \textit{complete bornological algebra} over $k$.  

\begin{ex}\label{borninj}
Let $A$ be a complete bornological algebra over $k$.  Denote the category of complete bornological $A$-modules by $\Mod_A(\CBorn_k)$.  The category of cochain complexes in $\Mod_A(\CBorn_k)$ will be denoted $\dg^{\CBorn_k}_A$.  The objects in this category will be called \textit{complete bornological dg-modules} over $A$.  The category $\CBorn_k$ of complete bornological vector spaces over $k$ is a Grothendieck quasi-abelian category by Proposition~\ref{bornisgqa} .  Therefore, by Proposition~\ref{modisgqa} the category $\Mod_A(\CBorn_k)$ is Grothendieck quasi-abelian.  It now follows from Proposition~\ref{injective} that the category $\dg_A^{\CBorn_k}$ of complete bornological dg-modules admits an injective model structure.  
\end{ex}

\begin{ex}
Let $A$ be a Banach algebra over $k$.  Denote the category of Banach $A$-modules by $\Mod_A(\Ban_k)$.  A chain complex in $\Mod_A(\Ban_k)$ will be a called a \textit{Banach dg-module} over $A$ and the category of Banach dg-modules over $A$ will be denoted $\dg_A^{\Ban_k}$.  The category of inductive systems in $\dg_A^{\Ban_k}$ will be denoted $\Ind(\dg_A^{\Ban_k})$.  The category $\Ban_k$ is a Grothendieck quasi-abelian category.  By Proposition~\ref{modisgqa}, the category $\Mod_A(\Ban_k)$ is Grothendieck quasi-abelian.  It then follows from Proposition~\ref{ind} that $\Ind(\Mod_A(\Ban_k))$ is also Grothendieck quasi-abelian.  Since there exists a canonical equivalence 
\[  \tu{Ch}(\Ind(\Mod_A(\Ban_k)))\ra\Ind(\dg_A^{\Ban_k})   \] 
of categories, it follows from Proposition~\ref{injective} that $\Ind(\dg_A^{\Ban_k})$ admits an injective model structure.
\end{ex}

According to Example~\ref{borninj} the category of complete bornological dg-vector spaces is endowed with an injective model structure.  However, to define a model category of algebras in these categories, it is important to introduce another model structure with the same weak equivalences which interacts well with the tensor product.  In order to do so, we introduce an extra assumption.  

An object $P$ in a quasi-abelian category $C$ is said to be \textit{projective} if the functor 
\[     \Hom_C(P,-):C^\circ\ra\Ab   \]
is exact.  The object $P$ is projective if and only if for every strict epimorphism $f:V\ra W$ in $C$, the induced map
\[   \Hom_C(P,V)\ra\Hom_C(P,W)  \]
is a surjection.  The category $C$ is then said to have \textit{enough projectives} if for every object $V$ in $C$, there exists a strict epimorphism $P\ra V$ with $P$ projective.

\begin{prop}\label{projective}
Let $C$ be a monoidal Grothendieck quasi-abelian category with enough projective objects and $R$ a monoid object in $C$.  There exists a left proper combinatorial model structure on the category $\dg_R^C$ of dg-modules over $R$ with the following classes of morphisms~:
\begin{itemize}
\item[$(\sF)$] The fibrations are the degreewise strict epimorphisms.
\item[$(\sW)$] The weak equivalences are the refined quasi-isomorphisms.
\item[$(\sC)$] The cofibrations are the those maps with the left lifting property with respect to trivial fibrations.
\end{itemize}
\end{prop}

\begin{proof}
One can use Jeff Smith's Theorem~1.7 in \cite{Be} or follow the same argument as in the abelian case, for example as in Proposition~7.1.2.8 of \cite{L2}, by checking the conditions in Proposition~A.2.6.13 of \cite{L1} replacing monomorphisms and epimorphisms using their strict notions and cohomology by refined cohomology.  We describe the latter.  

We define the cofibrations in this model structure to be the smallest weakly saturated class of morphisms containing generating cofibrations $\sC_0$ given by the collection of strict monomorphisms $\{C[u,n]\ra C[\id_R,n]\}_{n\in\bb{Z}}$ where $u:0\ra R$ (using Notation~\ref{notation}).  We know from Proposition~\ref{injective} that the class of refined quasi-isomorphisms is a perfect class and that the pushout of a cofibration along a refined quasi-isomorphism is a refined quasi-isomorphism.  To show that any morphism in $\dg_R^C$ which has the right lifting property with respect to every morphism in $\sC_0$ is a weak equivalence, one can use the argument in the abelian case in \cite{L2} using the refined cohomology.

It remains to show that the class of fibrations are indeed those stated in the proposition.  First assume that $f:M\ra N$ is a fibration and consider the diagram  
\[
\begin{tikzcd}
0  \arrow[r]\arrow[d,"g"]  & M\arrow[d,"f"]\\
C[\id_R,n] \arrow[r]  & N
\end{tikzcd}
\]
in $\dg^C_R$.  The inclusion $g$ is a trivial cofibration~: it is a refined quasi-isomorphism and can be realized as the composition of the morphism $g':0\ra C[u,n]$, which is an element in $\sC$ since there is a natural pushout diagram
\[
\begin{tikzcd}
C[u,n+1]  \arrow[r,"{i}"]\arrow[d]  & C[\id_R,n+1]\arrow[d]\\
0 \arrow[r,"{g'}"]  & C[u,n]
\end{tikzcd}
\]
in $\dg_R^C$ where $i\in\sC_0$, and the morphism $g'':C[u,n]\ra C[\id_R,n]$ in $\sC_0$.  Therefore, $g=g''\circ g'$ is a trivial fibration and since $f$ is a fibration by assumption, our diagram of interest has the right lifting property with respect to $g$.  Therefore $f$ is a degreewise epimorphism which is moreover strict since $R$ is a projective object.

Now assume that $f:M\ra N$ is a degreewise strict epimorphism.  For any $X,Y\in\dg^C_R$, consider the chain complex 
$    \Map(X,Y):=\{\Map(X,Y)^n,d^n\}_{n\in\bb{Z}} $
of abelian groups where 
\[  \Map(X,Y)^n=\prod_{p\in\bb{Z}}\Hom_C(X^p,Y^{p+n})  \] 
and 
\[  (d^nf)(x)=d_Y^n(f(x))-(-1)^nf(d^n_Xx)  \] 
for $f\in\Map(X,Y)^n$.  To show that $f$ is a fibration, it is enough to show that the map
\[    \phi:\Map(B,M)\ra\Map(B,N)\times_{\Map(A,N)}\Map(A,M)   \]
is surjective on $0$-cocycles for any trivial cofibration $g:A\ra B$.     

Since $g$ is a trivial cofibration, for all $n\in\bb{Z}$, the composition $A^n\ra B^n\ra Z^n$ is short exact where $Z^n=\tu{Coker}(g^n:A^n\ra B^n)$ for an object $Z$ in $\dg^C_R$.  It follows from Proposition~2.5 of \cite{CH} that $g$ is split and $Z$ is degreewise projective.  Therefore, the induced diagram
\[
\begin{tikzcd}
0 \arrow[r]  &\Map(Z,M) \arrow[r]\arrow[d,"\theta"] &\Map(B,M) \arrow[r]\arrow[d] &\Map(A,M) \arrow[r]\arrow[d] & 0\\
0 \arrow[r] &\Map(Z,N) \arrow[r] &\Map(B,N) \arrow[r] &\Map(A,N) \arrow[r] & 0
\end{tikzcd}
\]
is a diagram of exact sequences of cochain complexes of abelian groups.  The map $\theta$ is an epimorphism since $Z$ is degreewise projective and $f$ is a degreewise strict epimorphism.  A diagram chase, following the steps in Proposition~7.1.2.8 of \cite{L2}, then gives the result.
\end{proof}

The model structure of Proposition~\ref{projective} will be called the \textit{projective model structure}.

\begin{ex}\label{poind}
The category $\SNorm_k$ has enough projective objects by Proposition~3.2.11 of \cite{Sc}.  Let  $B_V=\{v\in V : \|v\|\leq 1\}$ be the unit ball for any vector space $V$.  For any $V$ in $\SNorm_k$, let $\bigoplus_{b\in B_V}k$ be endowed with the semi-norm $p$ defined by $p((c_b)_{b\in B_V})=\sum_{b\in B_V}p_b(c_b)$.  Then $\bigoplus_{b\in B_V}k$ is a projective object in $\SNorm_k$ and the morphism $\bigoplus_{b\in B_V}k\ra V$ sending $(c_b)_{b\in B_V}$ to $\sum_{b\in B_V}c_b b$ is a strict epimorphism.  A similar result holds for $\Norm_k$.

The category $\Ban_k$ has enough projective objects by Proposition~3.2.2 of \cite{P3}.  For any $V$ in $\Ban_k$, let $l_1(B_V,k)=\{ (c_b)_{b\in B_V} : c_b\in k,\sum_{b\in B_V}\|c_b\|_k<\infty\}$.  Then $l_1(B_V,k)$ is projective in $\Ban_k$ and the morphism $l_1(B_V,k)\ra V$ sending $(c_b)_{b\in B_V}$ to $\sum_{b\in B_V}c_b b$ is a strict epimorphism.

Using these results, it then follows from Corollary~1.4.14 of \cite{Sc} that the categories $\Ind(\SNorm_k)$, $\Ind(\Norm_k)$ and $\Ind(\Ban_k)$ have enough projectives.
\end{ex}
 
\begin{ex}
The category $\Born_k$ has enough projective objects by Proposition~2.13 of \cite{PS}.  The category $\SBorn_k$ has enough projective objects by Proposition~4.11 of \cite{PS}.  The category $\CBorn_k$ has enough projective objects by Proposition~5.8 of \cite{PS}.  These results can be deduced using the results of Example~\ref{poind}.
\end{ex}


Let $C$ be a closed symmetric monoidal Grothendieck quasi-abelian category.  Then the category $\dg^C$ of cochain complexes in $C$ is endowed with a symmetric monoidal structure 
\[    \otimes:\dg^C\times\dg^C\ra\dg^C   \]
given on objects $(M,d_M)$ and $(N,d_N)$ by the formula
\[     (M\otimes N)^n=\oplus_{p+q=n} M^p\otimes N^q   \]
and $d_{M\otimes N}(x\otimes y)=d_Mx\otimes y+(-1)^px\otimes d_Ny$ where $p$ is the degree of $x$.  The unit of $\dg^C$ is the unit of $C$ (concentrated in degree zero) and the symmetric stucture $M\otimes N\ra N\otimes M$ is that induced from $C$ with the sign convention $x\otimes y\mapsto (-1)^{pq}y\otimes x$.

The category $C$ is said to be a \textit{symmetric monoidal model category} if $C$ is endowed with a model structure such that the unit object with respect to the symmetric monoidal structure is cofibrant and the tensor product functor is a left Quillen bifunctor.

\begin{prop}\label{smmc}
Let $C$ be a closed symmetric monoidal Grothendeick quasi-abelian category and $R$ a commutative monoid object in $C$.  Then the category $\dg_R^C$ of dg-modules over $R$ is a symmetric monoidal model category with respect to the projective model structure.
\end{prop}

\begin{proof}
We need to show that the tensor product functor is a left Quillen bifunctor and that the unit object $R$ is cofibrant.  The second part is clear.  For the first part, since $\dg_R^C$ is combinatorial, if follows from Lemma~3.5 of \cite{SS} that it suffices to check the pushout product axiom on a pair of generating cofibrations and a pair of generating trivial cofibrations.

Let $f:C[u,n]\ra C[\id_R,n]$ and $g:C[u,m]\ra C[\id_R,m]$ be generating cofibrations where $u:0\ra R$.  Since the pushout product $f\wedge g$ of $f$ and $g$ is a pushout of the generating cofibration $h:C[u,m+n]\ra C[\id_R,n+m]$, the morphism $f\wedge g$ is a cofibration.  Now assume that $f$ and $g$ are trivial cofibrations.  Therefore, $h$ is a trivial cofibration and since the projective model structure is left proper, the morphism $f\wedge g$ is a weak equivalence.
\end{proof}

The tensor products we are interested in are the following.  Since $\Born_k$ is closed, the bornological tensor product satisfies the property that for bornological dg-vector spaces $E$ and $F$ there exists a bornological isomorphism
\[    \uHom(E\otimes F,G)\ra\uHom(E\times F,G)  \]
for any bornological dg-vector space $G$.  Equivalently, the functor 
\[  E\otimes-:\dg^{\Born_k}\ra\dg^{\Born_k}   \] 
is left adjoint to the functor $\uHom(E,-)$ between the category of bornological dg-vector spaces.  The fully faithful inclusion $\dg^{\CBorn_k}\ra\dg^{\Born_k}$ admits a left adjoint 
\[    {-}^c:\dg^{\Born_k}\ra\dg^{\CBorn_k}   \]
called the completion functor.

The \textit{complete tensor product} is the completion of the bornological tensor product.  The category of complete bornological dg-vector spaces is then closed under the complete tensor product.  It follows from Proposition~\ref{projective} that the category $\dg_A^{\Born_k}$ of bornological dg-modules over a bornological algebra and the category $\dg_A^{\CBorn_k}$ of complete bornological dg-modules over a complete bornological algebra are endowed with projective model structures.  Moreover, they are both closed symmetric monoidal model categories.  

\begin{lem}\label{cofibrant}
Every object in $\dg_A^{\Born_k}$ and $\dg_A^{\CBorn_k}$ is cofibrant with respect to the projective model structure.
\end{lem}

\begin{proof}
The category of bornological dg-modules is cocomplete.  Also, the class of generating cofibrations in the injective model structure are pushouts of coproducts of generating cofibrations in the projective model structure.  Any model category with the same generating cofibrations and weak equivalences define the same model structure.  Therefore the injective and projective model structure coincide for the category of bornological dg-modules.  It follows that every object is cofibrant.  The same argument applies to complete bornological dg-modules.
\end{proof}

We now consider some examples of symmetric monoidal functors between model categories in the quasi-abelian setting.

\begin{ex}
Let $\dg^{\Born_k}$ be endowed with its bornological tensor product and projective model structure.  When the category of dg-vector spaces is endowed with its canonical symmetric monoidal model structure, then the functor $\tu{Fine}$ of Example~\ref{fine} is symmetric monoidal functor between symmetric monoidal model categories.  
\end{ex}

\begin{ex}
Let $\dg^{\TVS_k}$ be endowed with its complete projective tensor product.  Then the bornologification functor $\tu{vN}$ and $\tu{pC}$ of Example~\ref{borcpct} are \textit{not} symmetric monoidal.  However, the restriction of $\tu{vN}$ to the subcategory of Banach dg-vector spaces and the restriction of $\tu{pC}$ to Fr\'echet dg-vector spaces are symmetric monoidal functors between model categories.
\end{ex}

\begin{ex}
By construction, the completion functors $t^\infty$, $s^\infty$ and $c^\infty$ are symmetric monoidal functors between symmetric monoidal model categories.  
\end{ex}

\begin{ex}\label{dgdiss}
There exists a differential graded \textit{dissection functor}
\[     \tu{diss}:\dg^{\Born_k}\ra\dg^{\Ind(\Ban_k)}     \]
which sends a complete bornological dg-vector space to an inductive system of dg-Banach spaces as in Section~\ref{qac}.  The dissection functor is fully faithful and its essential image consists of \textit{reduced} inductive systems, ie. those diagrams for which each map in the inductive system is a monomorphism \cite{M2}.  The left adjoint to the dissection functor sends an inductive system $(M_B)_{B\in\sD(M)}$ to $\colim_{B\in\sD(M)}M_B$.

Let $\Ind(\dg^{\Ban_k})$ be endowed with the canonical extension of the complete projective tensor product on $\dg^{\Ban_k}$ to $\Ind(\dg^{\Ban_k})$.  This makes $\Ind(\dg^{\Ban_k})$ a symmetric monoidal category and we endow it with a symmetric monoidal projective model structure.  Then the differential graded dissection functor $\tu{diss}$ is \textit{not} symmetric monoidal.  However, the composition 
\[  \tu{diss}\circ\tu{pC}:\dg^{\tu{Fre}_k}\ra\dg^{\Ind(\Ban_k)}   \] 
of the functor $\tu{pC}$ restricted to Fr\'echet dg-vector spaces with the differential graded dissection functor is a symmetric monoidal functor between symmetric monoidal model categories.
\end{ex} 

We conclude this section by showing that, although the functor in Example~\ref{dgdiss} is not an equivalence of categories, it is part of a Quillen equivalence of model categories with respect to the injective model structure.

\begin{prop}\label{snormborn}
There exists a Quillen equivalence
\[    \colim:\dg^{\Ind(\SNorm_k)}\rightleftarrows\dg^{\Born_k}:\tu{diss}  \]
of model categories endowed with the injective model structure.  This Quillen equivalence holds between ind-objects of normed spaces and separated bornological vector spaces and between ind-objects in Banach spaces and complete bornological vector spaces.
\end{prop}

\begin{proof}
We endow $\Ind(\dg^{\Ban_k})$ and $\dg^{\Born_k}$ with the injective model structure.  The left adjoint $\colim$ clearly preserves strict monomorphisms and refined quasi-isomorphisms.  Therefore the adjunction is a Quillen adjunction.  It remains to show that the left derived functor
\[      \bb{L}\colim:\h(\dg^{\Ind(\SNorm_k)})\ra\h(\dg^{\Born_k})       \]
is an equivalence of categories.  This follows from Proposition~3.10 of \cite{PS}.  The same is true utilizing Proposition~4.16 and Proposition~5.16 of \textit{loc.cit.} for the other two cases.
\end{proof}

\section{Differential graded algebras}

In this section we will set up the homotopy theory of differential graded algebras in a Grothendieck quasi-abelian category.  

Let $R$ be a commutative monoid object in $C$ and let $\tu{dga}_R^C$ denote the category of monoid objects in $\dg_R^C$.  The objects in $\tu{dga}_R^C$ will be called \textit{differential graded algebras} over $R$.  As in Section~\ref{dgm}, we will use the Grothendieck quasi-abelian category of complete bornological vector spaces as our primary example.  

\begin{ex}
Let $\tu{dga}^{\Born_k}_A$ denote the category of monoid objects in $\dg_A^{\CBorn_k}$ endowed with its complete bornological tensor product.  We call objects in this category \textit{complete bornological dg-algebras} over $A$.  Let $\cdga^{\CBorn_k}_A$ denote the category of commutative monoid objects in $\dg_A^{\CBorn_k}$ endowed with its complete bornological tensor product.  We call objects in this category \textit{commutative complete bornological dg-algebras} over $A$.  
\end{ex}

It follows from Proposition~\ref{smmc}, Lemma~\ref{cofibrant} and Theorem~4.1 of \cite{SS} that the category of (commutative) complete bornological dg-algebras is endowed with a combinatorial model structure induced from Proposition~\ref{projective}, ie. a fibration is a map if it is a fibration of complete bornological dg-modules and a weak equivalence if it is a weak equivalence of complete bornological dg-modules.  In the spirit of Section~\ref{dgm}, we prove a general theorem which encompasses this model structure, together with analogous structures on the examples in Proposition~\ref{snormgqa} and Proposition~\ref{bornisgqa}, by proving that a model structure exists on the category of algebras in an arbitrary Grothendieck quasi-abelian category.  

Let $C$ be a combinatorial symmetric monoidal model category and $D$ the collection of all morphisms in $C$ of the form
\[      x\otimes y\xra{\id_x\otimes g}x\otimes y'  \]
where $g$ is a trivial cofibration.  Let $\overline{D}$ denote the weakly saturated class of morphisms generated by $D$.  Recall that $C$ is said to satisfy the \textit{monoid axiom} if every morphism in $\overline{D}$ is a weak equivalence.

\begin{prop}\label{monoid}
Let $C$ be a Grothendieck quasi-abelian category and $R$ a commutative monoid object in $C$.  Let $\dg_R^C$ be endowed with the symmetric monoidal projective model structure.  Then $\dg_R^C$ satisfies the monoid axiom.
\end{prop}

\begin{proof}
It suffices to prove that every morphism in $\overline{D}$, the weakly saturated class of morphisms in $\dg_R^C$ generated by those of the form
$\id_M\otimes g:M\otimes N\ra M\otimes N'$ for $g$ a trivial cofibration, is a trivial cofibration in the projective model structure.  This in turn can be deduced if such morphisms in $D$ are trivial cofibrations.  Consider the strict exact sequence
\[    0\ra N\xra{g} N'\xra{g'} N''\ra 0     \]
which by definition means that it is exact in the usual sense and $g$ is strict.  Since the hom-functor $\Hom(-,Z)$ preserves strict exact sequences, the sequence
\[    0\ra M\otimes N\xra{\id_M\otimes g}M\otimes N'\xra{\id_M\otimes g'} M\otimes N''\ra 0  \]
is strict exact in $\dg^C_R$.  Therefore $\id_M\otimes g$ is a strict monomorphism.  Therefore we need to show that $M\otimes N''$ is an acyclic dg-module over $R$.  This follows from the fact that $N''$ admits a contracting homotopy.
\end{proof}

\begin{prop}
Let $C$ be a monoidal Grothendieck quasi-abelian category with enough projective objects and $R$ a commutative monoid object in $C$.  Then there exists a combinatorial model structure on the category $\tu{dga}_R^C$ of differential graded algebras over $R$ with the following classes of morphisms~:
\begin{itemize}
\item[$(\sF)$] A morphism is a fibration if and only if it is a fibration in $\dg_R^C$.
\item[$(\sW)$] A morphism is a weak equivalences if and only if it is a weak equivalence in $\dg_R^C$.
\item[$(\sC)$] The cofibrations are the those maps with the left lifting property with respect to trivial fibrations.
\end{itemize}
\end{prop}

\begin{proof}
This follows from Theorem~4.1 of \cite{SS}, Proposition~\ref{monoid} and Proposition~\ref{projective}.
\end{proof}

Let $C$ be a left proper combinatorial symmetric monoidal model category.  Recall that $C$ is said to be \textit{freely powered} if it satisfies the monoid axiom, the collection of cofibrations is generated by cofibrations between cofibrant objects and for every cofibration $f:x\ra y$ and every $n\geq 0$, the induced map 
\[      \wedge^n(f):(x\otimes y)\coprod_{x\otimes x}(x\otimes y)\ldots(x\otimes y)\coprod_{x\otimes x}(x\otimes y)\ra y^{\otimes n}  \]
(with $n$ factors of brackets in the domain) is a cofibration in the projective model category $C^{\Sigma_n}$ of functors from the symmetric group $\Sigma_n$ (on $n$ letters) to $C$.

\begin{prop}\label{freelypowered}
Let $C$ be Grothendieck quasi-abelian category with enough projective objects and $R$ a commutative monoid object in $C$ containing the field of rational numbers.  Let $\dg_R^C$ be endowed with the structure of a symmetric monoidal model category of Proposition~\ref{projective}.  Then $\dg_R^C$ is freely powered.
\end{prop}

\begin{proof}
It follows from Proposition~\ref{monoid} that $\dg_R^C$ satisfies the monoid axiom and we know that in the projective model structure, cofibrations between cofibrant objects generate the class of cofibrations.  It remains to check that for every cofibration $f:M\ra N$ in $\dg_R^C$, the map $\wedge^n(f)$ is a cofibration in $(\dg_R^C)^{\Sigma_n}$ for every $n\geq 0$.  It suffices to check this property for $f$ the generating cofibration $C[u,m]\ra C[\id_R,m]$ for some $m\in\bb{Z}$.

First note that $\wedge^n(f)$ is a pushout of the map $\phi:C[u,nm]\ra C[\id_R,nm]$ in the category $(\dg_R^C)^{\Sigma_n}$ where the (co)domain is endowed with the trivial representation (if $m$ is even) and the sign representation (if $m$ is odd).  Therefore it suffices to show that $\phi$ is a cofibration in $(\dg_R^C)^{\Sigma_n}$.  Owing to the Quillen adjunction
\[   -\otimes_R R[\Sigma_n]:\dg_R^C\rightleftarrows (\dg_R^C)^{\Sigma_n}:U   \]
where $R[\Sigma_n]$ is the regular representation of $\Sigma_n$, we have a retraction
\[
\begin{tikzcd}
C[u,nm]  \arrow[r]\arrow[d,"\phi"]  & C[u,nm]\otimes_R R[\Sigma_n] \arrow[r] \arrow[d]  & C[u,nm]\arrow[d,"\phi"]\\
C[\id_R,nm]  \arrow[r]  & C[\id_R,nm]\otimes_R R[\Sigma_n]    \arrow[r]        & C[\id_R,nm]
\end{tikzcd}
\]
in the category $(\dg_R^C)^{\Sigma_n}$ where $C[u,nm]\otimes_R R[\Sigma_n]\ra C[\id_R,nm]\otimes_R R[\Sigma_n]$ is a cofibration.  The result follows.
\end{proof}

Let $R$ be a commutative monoid object in $C$ and let $\cdga_R^C$ denote the category of commutative monoid objects in $\dg_R^C$.  The objects in $\cdga_R^C$ will be called \textit{commutative differential graded algebras} over $R$.

\begin{prop}\label{cdga}
Let $C$ be a symmetric monoidal Grothendieck quasi-abelian category and $R$ a commutative monoid object in $C$ containing the field of rational numbers.  Then there exists a combinatorial model structure on the category $\cdga_R(C)$ of commutative differential graded algebras over $R$ with the following classes of morphisms~:
\begin{itemize}
\item[$(\sF)$] A morphism is a fibration if and only if it is a fibration in $\dg_R^C$.
\item[$(\sW)$] A morphism is a weak equivalence if and only if it is a weak equivalence in $\dg_R^C$.
\item[$(\sC)$] The cofibrations are the those maps with the left lifting property with respect to trivial fibrations.
\end{itemize}
\end{prop}

\begin{proof}
This follows from Proposition~4.5.4.6 of \cite{L2}, Proposition~\ref{freelypowered} and Proposition~\ref{projective}.  See also \cite{Wh} where the result can be deduced (in this case $\dg_R^C$ satisfies the commutative monoid axiom).
\end{proof}

\begin{prop}
There exists a Quillen equivalence
\[    \cdga_A^{\Ind(\SNorm_k)}\overset{\colim}{\underset{\tu{diss}}\rightleftarrows}\cdga_A^{\CBorn_k}   \]
of model categories endowed with the projective model structure.  This Quillen equivalence holds between ind-objects of normed spaces and separated bornological vector spaces and between ind-objects in Banach spaces and complete bornological vector spaces.
\end{prop}

\begin{proof}
This follows from Proposition~\ref{snormborn}.
\end{proof}

\section{Chevalley-Eilenberg resolutions}

We now explain Chevalley-Eilenberg resolutions in the quasi-abelian setting.  Our main result shows that given a dg-Lie algebra in a quasi-abelian category over some commutative ring object, then the Chevalley-Eilenberg complex is a cofibrant resolution of the ring object with its trivial dg-Lie algebra structure.  

\begin{dfn}
Let $C$ be a monoidal quasi-abelian category and $R$ a commutative monoid object in $C$.  A \textit{dg-Lie algebra} over $R$ in $C$ is a dg-module $(\f{g},d)$ over $R$ equipped with a bracket
\[    [.,.]:\f{g}_p\otimes_R\f{g}_q\ra\f{g}_{p+q}  \]
satisfying the following conditions~:
\begin{enumerate}
\item For $x\in\f{g}_p$ and $y\in\f{g}_q$, the relation $[x,y]+(-1)^{pq}[y,x]=0$ holds.
\item For $x\in\f{g}_p$, $y\in\f{g}_q$ and $z\in\f{g}_r$, the relation 
\[ (-1)^{pr}[x,[y,z]]+(-1)^{pq}[y,[z,x]]+(-1)^{qr}[z,[x,y]]=0  \]
holds.
\item For $x\in\f{g}_p$ and $y\in\f{g}_q$, the relation $d[x,y]=[dx,y]+(-1)^{|x|+n}[x,dy]$ holds, ie. the differential $d$ on $\f{g}$ is a derivation with respect to the bracket.
\end{enumerate}
\end{dfn}

The category of dg-Lie algebras over $R$ in $C$ and bracket preserving maps of complexes will be denoted $\tu{dgLie}_R^C$.  We have an obvious forgetful functor
\[    f:\tu{dgLie}_R^C\ra\dg_R^C   \]
to the category of dg-modules over $R$.  

Every associative dg-algebra $A$ has a primordial dg-Lie algebra structure with Lie bracket 
\[   [.,.]:A_p\otimes_k A_q\ra A_{p+q}  \]
given by $[x,y]=xy-(-1)^{pq}yx$.  The left adjoint 
\[     U:\tu{dgLie}_R(C)\ra\tu{dga}_R(C) \]
to the forgetful functor associates to $\f{g}$ its \textit{universal enveloping algebra} $U(\f{g})$.  This algebra has the following concrete form.  Let $M$ be a dg-module over $R$ and denote by 
\[     T:\dg_R(C)\ra\tu{dga}_R(C) \]
the left adjoint to the forgetful functor which associates to $M$ the tensor algebra $T(M)$ of $M$.  Explicitly, $T(M)=\bigoplus_{n\geq 0} M^{\otimes n}$ where $M^0$ is $R$ in degree $0$ by convention.  Then $U(\f{g})$ is the quotient of $T(\f{g})$ by the two-sided ideal generated by the relations $[x,y]=x\otimes y-(-1)^{pq}(y\otimes x)$ where $x\in\f{g}_p$ and $y\in\f{g}_q$.  

One can endow $U(\f{g})$ with a filtration 
\[ U(\f{g})^{\leq 0}\xhookrightarrow{i_0}U(\f{g})^{\leq 1}\xhookrightarrow{i_1}U(\f{g})^{\leq 2}\xhookrightarrow{i_2}\ldots  \] 
where $U(\f{g})^{\leq n}$ is the image of $\oplus_{0\leq i\leq n}\f{g}^{\otimes i}$ in $U(\f{g})$.  The associated graded dg-algebra $\tu{gr}(U(\f{g}))$ of this filtered object is given by $\tu{gr}(U(\f{g})):=\oplus_{n\in\bb{N}}U(\f{g})_n$ for which the underlying complex of $U(\f{g})_n$ is $\coker(i_n)$.  This graded dg-algebra is moreover a graded commutative dg-algebra.  Therefore the following result holds where 
\[    \Sym_R:\tu{dgLie}_R^C\ra\tu{cdga}_R^C  \] 
is left adjoint to the forgetful functor.

\begin{prop}(Quasi-abelian Poincar\'e-Birkhoff-Witt).
Let $C$ be a monoidal locally presentable quasi-abelian category and $R$ a commutative monoid object in $C$ containing the rational numbers.  Let $\f{g}$ be a dg-Lie algebra over $R$.  Then there exists an isomorphism
\[    \Sym_R(\f{g})\ra\tu{gr}(U(\f{g}))  \]
of commutative dg-algebras over $R$.
\end{prop}

\begin{proof}
The map $f:\Sym_R(\f{g})\ra\tu{gr}(U(\f{g}))$ is induced from the canonical morphism $\f{g}\ra U(\f{g})^{\leq 1}$ and one can follow the general proof given in Section~1.3.7 of \cite{DM}.
\end{proof}

In particular, there exists an isomorphism $U(\f{g})\simeq\Sym_R(\f{g})$ of dg-modules over $R$.

\begin{prop}\label{dglie}
Let $C$ be a monoidal Grothendieck quasi-abelian category and $R$ a commutative monoid object in $C$ containing the field of rational numbers.  Then there exists a combinatorial model structure on the category $\tu{dgLie}_R^C$ of differential graded Lie algebras over $R$ with the following classes of morphisms~:
\begin{itemize}
\item[$(\sF)$] A morphism is a fibration if and only if it is a fibration in the projective model structure on $\dg_R^C$.
\item[$(\sW)$] A morphism is a weak equivalences if and only if it is a weak equivalence in the projective model structure on $\dg_R^C$.
\item[$(\sC)$] The cofibrations are the those maps with the left lifting property with respect to trivial fibrations.
\end{itemize}
\end{prop}

\begin{proof}
Let $g:\dg_R^C\ra\tu{dgLie}_R^C$ denote the free functor which is left adjoint to the forgetful functor $f$.  We then define the collection of generating cofibrations $\sC_0$ to be $\{g(C[u,n+1])\ra g(C[\id_R,n])\}_{n\in\bb{Z}}$ using Notation~\ref{notation}.  It suffices to show that our classes of generating cofibrations and weak equivalences satisfy the conditions in Proposition~A.2.6.13 of \cite{L1}.  Moreover, one must prove that a morphism in $\tu{dgLie}_k^C$ is a fibration if and only if it is a degreewise strict epimorphism.  One can follow the proof in Proposition~2.1.10 of \cite{L3} using the maps in the projective model structure of Proposition~\ref{projective}.
\end{proof}

We will call the model structure of Proposition~\ref{dglie} the \textit{projective model structure}.  For the remainder of this paper we will use the notation $\bigwedge\f{g}:=\Sym(\f{g}[1])$.

\begin{prop}\label{cofibrep}
Let $C$ be a monoidal Grothendieck quasi-abelian category and $R$ a commutative monoid object in $C$ containing the field of rational numbers.  Let $\f{g}$ be a dg-Lie algebra over $R$.  Then $U(\f{g})\otimes\bigwedge\f{g}$ is a cofibrant replacement for $R$ in the model category $\Mod(U(\f{g})):=\Mod_{U(\f{g})}(\dg_R(C))$ of $U(\f{g})$-modules.  
\end{prop}

\begin{proof}
Define the \textit{cone} $\f{g}\oplus\f{g}[1]$ of $\f{g}$ to be the dg-Lie algebra with differential 
\[    d_n(x+\epsilon y):=dx+y-\epsilon dy  \]
and Lie bracket
\[     [x+\epsilon y,x'+\epsilon y']=[x,x']+\epsilon([y,x']+(-1)^n[x,y'])   \]
for any $x\in\f{g}_n$ and $y\in\f{g}_{n+1}$.  The underlying dg-module of $\f{g}\oplus\f{g}[1]$ is the mapping cone of the identity on $\f{g}$.  Therefore $\f{g}\oplus\f{g}[1]$ is contractible and thus there exists a reduced quasi-isomorphism $0\ra\f{g}\oplus\f{g}[1]$ in $\dg_R(C)$.  Since $U$ preserves reduced quasi-isomorphisms, the map $U(0)=R\ra U(\f{g}\oplus\f{g}[1])$ is a reduced quasi-isomorphism.
By the Poincar\'e-Birkhoff-Witt theorem, there exists an equivalence $U(\f{g}\oplus\f{g}[1])\simeq\bigwedge\f{g}\otimes_kU(\f{g})$ in $\Mod(U(\f{g}))$ and the result follows.
\end{proof}

\begin{ex}
Let $C$ be a closed symmetric monoidal locally presentable quasi-abelian category and $A$ an object in $\cdga_R^C$.  Assume $A$ is endowed with an action of a dg-Lie algebra $\f{g}$ over $R$.  Then a model for the quotient space $A/\f{g}$ is the dg-algebra $\uHom_{\f{g}}(R,A)$.  Then we have a chain
\[     \uHom_{\f{g}}(R,A)\simeq \uHom_{U(\f{g})}(R,A)\simeq \uHom_{U(\f{g})}(\bigwedge\f{g}\otimes_RU(\f{g}),A)\simeq\uHom_{R}(\bigwedge\f{g},A)\hookleftarrow\bigwedge\f{g}'\otimes_RA   \]
of equivalences.  
\end{ex}

\begin{ex}
The category $\tu{dgLie}_A^{\CBorn_k}$ of \textit{complete bornological dg-Lie algebras} admits a projective model structures due to Proposition~\ref{projective}.  Given a complete bornological dg-Lie algebra $\f{g}$ over a complete bornological algebra $A$, then the Chevalley-Eilenberg complex associated to $\f{g}$ is a cofibrant replacement of $A$ in the model category of modules over the universal enveloping algebra of $\f{g}$.  Moreover, the inclusion $\bigwedge\f{g}'\otimes_kA\hookrightarrow\uHom_{k}(\bigwedge\f{g},A)$ is dense when $\bigwedge\f{g}$ satisfies the bornological approximation property \cite{KM}.  This is satisfied, for example, when $\bigwedge\f{g}$ is a nuclear Fr\'echet space. 
\end{ex}

\section{Koszul-Tate resolutions}

Let $C$ be a symmetric monoidal quasi-abelian category and $R$ a commutative monoid object in $C$.  For any dg-$R$-module $P$, we will define an object $K(R,P)$ which we prove is a cofibrant replacement of $R$ in a certain model category of commutative dg-algebras.  This dg-algebra is closely related to the Koszul complex (associated to an element in $P$) in the setting of dg-algebras in the quasi-abelian setting.

We endow $\dg_R$ with the symmetric monoidal model structure of Proposition~\ref{projective} and define the dual of a $R$-dg-module $P$ to be
$  P':=\uHom(P,R)  $
where $R$ is the unit object of the monoidal structure.  

Denote by
\[    Q:=\Sym(P')  \]
the symmetric algebra on $P'$.  Let $\bigwedge P'$ be the exterior algebra of $P'$ considered as a graded $R$-dg-algebra.  The $Q$-dg-module $Q\otimes_R\bigwedge P'$ is then a (non-positively) graded commutative $Q$-dg-algebra where the grading is given by 
\[   (Q\otimes_R\bigwedge P')_n:=Q\otimes_R{\bigwedge}^{-n}P'  \]
for $n\leq 0$.  

We endow this graded commutative $Q$-dg-algebra with a differential as follows.  Firstly, consider the map
\[   h_{n+1}:Q\otimes_R{\bigwedge}^{n+1} P'\ra Q\otimes_R P'\otimes_R P\otimes_R{\bigwedge}^{n+1} P' \]
induced from the canonical map $i:R\ra\uHom(P,P)\simeq P'\otimes_R P$ sending $1_R$ to $\id_P$.  Secondly, consider the map
\[   c_{n+1}:Q\otimes_R P'\otimes_R P\otimes_R{\bigwedge}^{n+1} P' \ra Q\otimes_R{\bigwedge}^{n} P' \]
induced by the action of $P'\subset Q$ on $Q$ by right multiplication and the action of $P\subset{\bigwedge}^{n+1}P$ on ${\bigwedge}^{n+1} P'$ defined by $m\cdot(f_1\wedge\ldots\wedge f_{n+1}):=\sum_i (-1)^i f_i(m)(f_1\wedge\ldots\wedge\hat{f}_j\wedge\ldots\wedge f_{n+1})$ for $m\in P$.  We then define the differential
\[   d_{n+1}:Q\otimes_R{\bigwedge}^{n+1}P'\ra Q\otimes_R{\bigwedge}^nP'  \]
by $d_{n+1}:=c_{n+1}\circ h_{n+1}$.  We denote by 
\[   K(R,P):=(Q\otimes_R\bigwedge P',d)   \] 
the resulting commutative dg-algebra over $Q$.

The zero section $R\ra Q$ defines a natural map given by the composition
\[ K(R,P)\ra (Q\otimes_R\bigwedge P')_0=Q\ra R  \]
which we call the augmentation map.  

\begin{prop}\label{koscorep}
The augmentation map $K(R,P)\ra R$ is a cofibrant replacement of $R$ in the model category $\cdga_Q$ of commutative dg-algebras over $Q$ with respect to the projective model structure.
\end{prop}

\begin{proof}
Since every object is cofibrant, it suffices to check that $K(R,P)\ra R$ is a trivial fibration.  It is clearly a fibration so we will check that it is a weak equivalence.  By definition, we need to show that 
\[  H^n(K(R,P))\ra H^n(R)   \]
is an isomorphism of objects in $C$.  The underlying dg-module of $P'\oplus P'[1]$ is the mapping cone of the identity on $P'$.  Therefore $0\ra P'\oplus P'[1]$ is a reduced quasi-isomorphism.  A map between dg-modules $M\ra N$ over $R$ is a reduced quasi-isomorphism if and only if $\Sym_R(M)\ra\Sym_R(N)$ is a quasi-isomorphism.  Therefore
\[    \Sym_R(0)=R\ra\Sym_R(P'\oplus P'[1])\simeq K(R,P)  \]
is a reduced quasi-isomorphism and the result follows.
\end{proof}

The usual Koszul algebra is 
\[  K(R,P;m):=(\bigwedge P',d_m)  \] 
which is a commutative dg-algebra over $R$ for any $m\in P$.  The differential $d_m$ is induced by contraction along $m$ where $d_m(f_1\wedge\ldots\wedge f_n)=\sum_{i=1}^{n}(-1)^{i+1} f_i(m)(f_1\wedge\ldots\wedge\hat{f}_i\wedge\ldots\wedge f_n)$.  This choice of element $m$ also induces a map $Q\ra R$ given by evaluation at $m$.  We denote by $R_m$ the dg-algebra $R$ (concentrated in degree $0$) with this $Q$-algebra structure.  

\begin{lem}\label{alemma}
There exists an equivalence
\[  R_m\otimes_Q K(R,P)\ra K(R,P;m)   \]
of dg-algebras over $Q$.
\end{lem}

\begin{proof}
We view $\bigwedge P'$ as a $Q$-module via the composite morphism $Q\ra R\ra\bigwedge P'$ of dg-algebras where the first map is given by evaluation at $m$ and the second map is the canonical one.  The underlying graded $Q$-module of $R_m\otimes_Q K(R,P)$ is then $R_m\otimes_Q Q\otimes_R\bigwedge P'$ which is isomorphic to $\bigwedge P'$ as graded modules over $Q$.  It is also clear that the induced differential on $R_m\otimes_Q K(R,P)$ is $d_m$ and the result follows.            
\end{proof}

Lemma~\ref{alemma} extends to an equivalence of commutative dg-algebras over $R$ via the canonical map $R\ra Q$. 

\begin{ex}
We denote by $\tu{dAff}^{\CBorn_k}:=(\cdga^{\CBorn_k})^\circ$ the opposite of the category of complete bornological commutative dg-algebras over $k$.  For any morphism $f:A\ra B$ of complete bornological commutative dg-algebras, there exists a relative cotangent complex $\bb{L}_{A/B}$ (see \cite{L2}).  The morphism $f:A\ra B$ is said to be \'etale if it is of finite presentation and $\bb{L}_{A/B}\simeq 0$.

Recall that a model site is a model category together with a Grothendieck topology on its homotopy category \cite{TV}.  The category $\tu{dAff}^{\CBorn_k}$ is a model site for the \'etale topology.  Therefore, we obtain a model category $\tu{St}(\tu{dAff}^{\CBorn_k})$ of complete bornological stacks on the model site for the \'etale topology (see \textit{loc.cit.} for more details).
 
Let $A$ be a complete bornological dg-algebra over $k$ and $X=\Spec A$ the image of $A$ in $\tu{St}(\Aff^{\CBorn_k})$.  Set $T^*X:=\Spec(\Sym(\bb{T}_X))$ where $\bb{T}_X:=\bb{T}_A$ is the tangent complex (dual to the cotangent complex) of $A$.  Let $df:X\ra T^*X$ be the differential for a function $f$ on $X$.  Then the stack $\tu{dCrit}(f)$ given by the homotopy pullback
\[
\begin{tikzcd}
X  \arrow[r,"{df}"]  &T^*X  \\
\tu{dCrit}(f)  \arrow[r]\arrow[u]     & X\arrow[u,"0"]
\end{tikzcd}
\]
in the model category of complete bornological stacks, called the \textit{derived critical locus} of $f$, can be calculated explicitly as follows.  There exists a chain of equivalences
\[  \tu{dCrit}(f)\simeq X\times_{\{df,T^*X,0\}} QX\simeq X\times_{\{df,T^*X,0\}}\Spec(K(A,\bb{L}_X))\simeq\Spec(K(A,\bb{L}_X;df))   \]
of complete bornological stacks by Proposition~\ref{koscorep} and Lemma~\ref{alemma}.  Therefore, the shifted cotangent stack
\[   T^*X[-1]:=\Spec (\Sym(\bb{T}_X[1]),d_{df})  \]
is a model for the derived critical locus of $f$ and $(\Sym(\bb{T}_X[1]),d_{df})$ is a model for the complete bornological commutative dg-algebra of functions on the derived critical locus.
\end{ex}

\newpage


\end{document}